\documentclass[reqno,12pt]{amsart}
\usepackage{amsmath,amsthm,enumerate,cite,graphics,pstricks,amsfonts,latexsym,amsopn,verbatim,amscd,amssymb}
\usepackage{color}
\usepackage{psfrag}
\usepackage[all]{xy}

\theoremstyle{plain}
\newtheorem{thm}{Theorem}[section]

\newtheorem{lem}[thm]{Lemma}
\newtheorem{cor}[thm]{Corollary}
\newtheorem{prop}[thm]{Proposition}

\theoremstyle{definition}

\newtheorem{rem}[thm]{Remark}

\newtheorem{que}[thm]{Question}

\newcommand{\C}{\mathcal{C}}

\DeclareMathOperator{\GL}{GL}
\DeclareMathOperator{\Cent}{Cent}

\begin{document} 

\title[On groups with a given central factor]{On groups with a given central factor} 

\author[S. J. Baishya  ]{Sekhar Jyoti Baishya} 
\address{S. J. Baishya, Department of Mathematics, Pandit Deendayal Upadhyaya Adarsha Mahavidyalaya, Behali, Biswanath-784184, Assam, India.}

\email{sekharnehu@yahoo.com}

\begin{abstract}
We have classified, upto isoclinism, certain groups with a given central factor. As an application, we classify, upto isoclinism, groups having at the most nine element centralizers. Among other results of independent interest, we have classified, upto isoclinism, groups having a central factor of order $p^3$, $p$ a prime. All these improves some previous results. 
\end{abstract}

%\thanks{*This paper is a part of my Ph.D thesis.}
\maketitle

\section{Introduction} \label{S:intro}

A group $K$ is said to be capable if  $K \cong \frac{G}{Z(G)}$ for some group $G$. The study of capable groups was initiated by R. Baer \cite{baer}, who determined all capable groups which are direct sums of cyclic groups. Consequently, all capable finite abelian groups are characterised by R. Baer's result. For finite $p$-groups ($p$ a prime)  capability is closely related to their classification. The authors of \cite{capable} characterised capable extraspecial groups (only $D_8$ and the extraspecial groups of order $p^3$ and exponent $p$ are capable); they also studied the metacyclic capable groups.

Our approach in this paper is towards the opposite direction. To be specific we would like to explore the following question:

\begin{que} \label{q6}
Given a capable group $K$, characterize $G$ such that $K \cong \frac{G}{Z(G)}$.
\end{que} 

It is worth mentioning that the study of groups with a given central factor have been considered by many researchers. See for example \cite{machale, amiriF, amiriS, baishya2, heffernan} etc. for advances on this and related concepts.

Here we have considered the above question and classified, upto isoclinism, certain groups with a given central factor. As an application, we classify, upto isoclinism, groups having at the most nine element centralizers. It may be mentioned here that M. Chipola \cite{cipolla} in 1909 proved that a non-abelian group contains minimum four element centralizers; and it has four element centralizers if and only if its central factor is of order four. Among other results of independent interest, we have classified, upto isoclinism, groups having a central factor of order $p^3$, $p$ a prime. All these improves some previous results.

In this paper  $Z(G), G', \Cent(G)$, $nacent(G)$ and $\omega (G)$ denotes the  center, commutator subgroup, set of element centralizers, set of non-abelian element centralizers and the size of a maximal set of pairwise non-commuting elements of $G$ respectively.  $C(x)$ denotes the centralizer of $x \in G$ and $C_n$ denotes the cyclic group of order $n$.

\section{Preliminary  results}

We begin with the notion of isoclinism between two groups introduced by P. Hall \cite{hall} in 1940. Two groups $G$ and $H$ are said to be isoclinic if there are two isomorphisms $\varphi : G/Z(G) \longrightarrow H/Z(H)$ and $\phi : G' \longrightarrow H'$ such that if 
\[
\varphi(g_1Z(G))=h_1Z(H) \;\; \text{and} \;\; \varphi(g_2Z(G))=h_2Z(H)
\]
with $g_1, g_2 \in G, h_1, h_2 \in H$, then 
\[
\phi([g_1, g_2])=[h_1, h_2].
\] 

Isoclinism is an equivalence relation weaker than isomorphism having many family invariants. See \cite{baishyaS} for some invariants concerning element centralizers.

A group $G$ is said to be $ \C_n$-group if $\mid \Cent(G)\mid=n$ and  primitive $ \C_n$-group if $\mid \Cent(G)\mid= \mid \Cent(\frac{G}{Z(G)})\mid=n$. 

\begin{rem}\label{remark}
In view of \cite[Theorem 3.1]{non}, we have $G$ is a $C_n$-group if and only if $Z(G)$ has finite index in $G$. Moreover, by a classical result of Schur \cite[Proposition 2.4.4]{schur}, we know that if $Z(G)$ has finite index in $G$, then $G'$ is finite. In the present scenario, using \cite[Hall, p. 134]{hall}, every $C_n$ group $G$ is isoclinic with a finite group  $M$ such that $Z(M) \subseteq M'$. Furthermore, we have $\mid G' \cap Z(G) \mid=\mid M' \cap Z(M) \mid$.
\end{rem}

 In \cite{en09} it is proved that if $G$ is a finite group with  $G' \cap Z(G)=\lbrace 1 \rbrace$, then it is primitive $\C_n$-group. The following proposition  generalizes this result. It also shows that the converse of this statement need not be true.

\begin{lem}\label{sss}
For any $ \C_n$-group $G$ and $N\unlhd G$, we have $\mid \Cent(\frac{G}{N})\mid= \mid \Cent(\frac{G}{G' \cap N})\mid$. 
\end{lem}

\begin{proof}
It follows combining \cite[p. 134]{hall} with \cite[Lemma 3.2]{non}. 
\end{proof}

\begin{prop}\label{b1136}
For any $ \C_n$-group $G$ and $N\unlhd G$ with  $G' \cap N=\lbrace 1 \rbrace$, we have  $\mid \Cent(G)\mid= \mid \Cent(\frac{G}{N})\mid$. The converse need not be true.  
\end{prop}

\begin{proof}
It follows from Lemma \ref{sss}. However, $G=GL(2, 3)$ is a primitive $\C_{14}$-group with $ \mid G' \cap Z(G) \mid = 2$.
\end{proof}

A group $G$ is called a CA-group if $C(x)$ is abelian for all $x \in G \setminus Z(G)$.
Note that $GL(2, 3)$ is a primitive $\C_{14}$-CA-group whose central factor is not a CA-group.

\begin{prop}\label{b1}
Let $G$ be any primitive $\C_{n}$-CA-group such that $ \frac{G}{Z(G)}$ is not a CA-group. Then $G' \cap Z(G) \neq \lbrace 1 \rbrace$.
\end{prop}

\begin{proof}
Combining  \cite[p. 134]{hall} and \cite[Proposition 2.1]{baishyaS}, we get the result.
\end{proof}

The following result generalizes \cite[Corollary 3.2]{baishyaS}.

\begin{prop}\label{SAS}
Let $G$ be any $\C_{n}$-group such that $Z(G')$ is trivial. Then $G$ is isoclinic with $\frac{G}{Z(G)}$. In particular, $G$ is a primitive $\C_{n}$-group.
\end{prop}

\begin{proof}
It follows from \cite[p. 135]{hall} and Proposition \ref{b1136}, by noting that in the present scenario, we have  $G' \cap Z(G)=\lbrace 1 \rbrace$. 
\end{proof}

We now give another class of primitive $\C_{n}$-groups. 

\begin{prop}\label{CG17}
Let $G$ be any $\C_{n}$-group such that $\frac{G}{Z(G)}=\frac{K}{Z(G)} \rtimes \frac{H}{Z(G)}$ is a Frobenius group with $K$ and $H$ abelian. Then $\mid \Cent(G)\mid= \mid \Cent(\frac{G}{Z(G)})\mid=\mid G' \mid+2$. 
\end{prop}

\begin{proof}
Without any loss we may assume that $G$ is a finite group \cite{non}. Now, the result follows from the definition of Frobenius group and \cite[Proposition 3.8]{baishyaS}. 
\end{proof}

The following key result helps in determining the structure of any $\C_{n}$-group having prime order commutator subgroup.

\begin{prop}\label{n3}
Let $G$ be any $\C_{n}$-group such that $\mid G' \mid=p$, $p$ a prime. 
\begin{enumerate}
	\item If $\frac{G}{Z(G)}$ is non-abelian, then $G$ is isoclinic with $\frac{G}{Z(G)}$. In particular, $G$ is a primitive $\C_{n}$-group.
	\item If $\frac{G}{Z(G)}$ is abelian, then $G$ is isoclinic with an extraspecial $p$-group of order $p^{2a+1}$ for some $a$. In particular,  $n =\frac{p^{2a}-1}{p-1}+1$.
\end{enumerate} 
\end{prop}

\begin{proof}

Without any loss we may assume that $G$ is a finite group \cite{non}. Now,

a) In this case we have $G' \cap Z(G)=\lbrace 1 \rbrace$ and hence by \cite[p. 135]{hall}, $G$ is isoclinic with $\frac{G}{Z(G)}$. Second part follows from Proposition \ref{b1136}.

b) It follows from \cite[Proposition 3.24]{baishyaF}.
\end{proof}

For finite groups $G$ with $\mid G' \mid=2$, we have the following specific result.

\begin{prop}\label{ncF2}
Let $G$ be any finite group. The following are equivalent.
\begin{enumerate}
	\item $\mid G'\mid=2$.
	\item $G=C(x)\cup C(y)\cup C(xy)$ for any $x, y \in G$ with $[x, y] \neq 1$.
	\item $G=A \times P$, where $A$ is  abelian and  $\mid P' \mid=2$.
\end{enumerate} 
\end{prop}

\begin{proof}

(a) $\Longrightarrow$ (b)
Let $g \in G \setminus (C(x)\cup C(y))$. If $g \notin C(xy)$, then $[g, xy]=[g, x]=[g, y]$, which is impossible. 

(b) $\Longrightarrow$ (c)
Let $a \in G \setminus Z(G)$ be any element. Suppose $b \in G \setminus C(a)$. Then $G=C(a)\cup C(b)\cup C(ab)$; which implies $\mid C(a)\mid=\frac{\mid G \mid}{2}$. Therefore $G$ is isoclinic with an extraspecial $2$-group \cite{ish1}. Consequently, by Ito \cite{ito}, $G=A \times P$, where $A$ is abelian  and $\mid P' \mid=2$. 

(c) $\Longrightarrow$ (a) Clearly.
\end{proof}

\begin{cor}\label{ncF31}
Let $G$ be any  $\C_{n}$-group. If $\mid G'\mid=2$, then $\mid \Cent(G)\mid = \mid \frac{ G }{Z(G)} \mid$.
\end{cor}

\begin{proof}
Without any loss we may assume that $G$ is a finite group \cite{non}. Now, the result follows from \cite[Proposition 2.21]{baishyarima}.
\end{proof}

 Recall that a group $G$ is called an F-group if $C(x) \subseteq C(y)$ implies $C(x)=C(y)$ for any $x, y \in G \setminus Z(G)$. In the following result, which improves \cite[Theorem 3]{nca}, $Z(x)$ denotes the center of $C(x)$ for any $x \in G$.

\begin{prop}\label{ncF1}
Let $G$ be a finite non-abelian $\C_{n}$-group. The following are equivalent.

\begin{enumerate}
	\item $G$ is an F-group and $n= \mid \frac{G}{Z(G)}\mid$.
	\item $G=A \times P$, where $A$ is an abelian subgroup of odd order and $P$ is a Sylow $2$-subgroup which is an F-group with $\mid \Cent(P) \mid = \mid \frac{P}{Z(P)}\mid$.
	\item $G=A \times P$, where $Z(x)=Z(P) \sqcup xZ(P)$ for all $x \in P \setminus Z(P)$.
\end{enumerate} 
\end{prop}

\begin{proof}

(a) $\Longrightarrow$ (b) Suppose $n= \mid \frac{G}{Z(G)}\mid$. Then by \cite[Theorem 3.8]{baishyaF}, $G=A \times P$, where $A$ is an abelian subgroup  of odd order and $P$ is a Sylow $2$-subgroup  with $\mid \Cent(P) \mid = \mid \frac{P}{Z(P)}\mid$. Note that $G$ is isoclinic with $P$. Consequently, using \cite[Proposition 3.4]{baishyaF} and \cite[Proposition 2.6]{baishyarima}, we have $P$ is an F-group.

(b) $\Longrightarrow$ (c) It follows from  \cite[Proposition 3.4]{baishyaF}.

(c) $\Longrightarrow$ (a) It follows from  \cite[Proposition 3.4]{baishyaF}.
\end{proof}

The following result generalizes \cite[Proposition 2.7]{baishya2}. Recall that a finite group $G$ is said to be a $Z$-group is every Sylow subgroup of $G$ is cyclic.

\begin{prop}\label{co166}
Let $G$ be any $\C_{n}$-group such that $\frac{G}{Z(G)}$ is a $Z$-group. Then $G$ is isoclinic with $\frac{G}{Z(G)}$. In particular, $G$ is a primitive $\C_{n}$-group.
\end{prop}

\begin{proof}
Without any loss we may assume that $G$ is a finite group \cite{non}.
Moreover, by \cite[Corollary 2.1.3 and Proposition 2.1.7]{schur}, we have $G' \cap Z(G)=\lbrace 1 \rbrace$. Now the result follows from \cite[p. 135]{hall}. Last part follows from \cite[Lemma 3.2]{non}.
\end{proof}

In view of \cite[Corollary 2.5]{baishya}, we now have the followig result.

\begin{cor}\label{n3987}
Let $G$ be any group such that $\mid \frac{G}{Z(G)}\mid=pq$, $p< q$ are primes. Then $G$ is isoclinic with $\frac{G}{Z(G)}$. In particular, $G$ is a primitive $\C_{q+2}$-group. 
\end{cor}

Recall that a group $G$ is called perfect if $G=G'$.

\begin{lem}\label{nis}
Let $G$ be any group such that $G' \cap Z(G) = \lbrace 1 \rbrace$. Then $\frac{G}{Z(G)}$ is perfect  if and only if $G=Z(G) \times G'$. 
\end{lem}

\begin{proof}
The result follows from \cite[Proposition 2.25]{baishyarima}.
\end{proof}

In the following result, we provide some additional information to \cite[Corollary 2.1]{amiriS}, where $Sz(q)$ denotes the Suzuki group over the field with $q$ elements.

\begin{prop}\label{ssd}
Let $G$ be any group such that  $\frac{G}{Z(G)} \cong Sz(q)$, $q \neq 8$. Then $G$ is isoclinic with $G'$ and $\frac{G}{Z(G)}$. In particular,  $G=Z(G) \times G'$,  $ \mid \Cent(G) \mid= \mid \Cent(G') \mid= \mid \Cent(\frac{G}{Z(G)}) \mid$ and $\omega(G)=\omega(G')=\omega(\frac{G}{Z(G)})$. 
\end{prop}

\begin{proof}
Using \cite[Theorem 7.4.2 and Proposition 2.1.7]{schur}, Lemma \ref{nis} and \cite[Proposition 2.25]{baishyarima}, we get the result.
\end{proof}

\begin{lem}\label{c1234}
Let $G$ be any finite group such that $\frac{G}{Z(G)} \cong D_{2n}$. Then $\mid G' \cap Z(G) \mid=1$ if $n$ is odd and $\mid G' \cap Z(G) \mid=2$ if $n$ is even. In particular, $G$ is isoclinic with $\frac{G}{Z(G)}$ if and only if $n$ is odd.
\end{lem}

\begin{proof}
In view of \cite[Propositions 2.1.7 and 2.11.4]{schur}, we have $\mid G' \cap Z(G) \mid=1$ if $n$ is odd and $\mid G' \cap Z(G) \mid=2$ if $n$ is even. Second part follows from \cite[p. 135]{hall}.
\end{proof}

\begin{cor}\label{coro}
Let $G$ be any finite group such that $\frac{G}{Z(G)} \cong D_{2n}$.
The following are equivalent.
\begin{enumerate}
    \item $n$ is odd.
	\item $G$ is primitive $\C_n$.
	\item $G$ is isoclinic with $\frac{G}{Z(G)}$.
	\item $\mid G' \cap Z(G) \mid=1$.
\end{enumerate} 
\end{cor}

\begin{proof}
 (a) $\Longrightarrow$ (b) It follows combining Lemma \ref{c1234} and Proposition \ref{b1136}.

 (b) $\Longrightarrow$ (c) If $G$ is primitive $\C_n$, then  (c) follows combining  Lemma \ref{c1234} and  \cite[Corollary 2.4, Proposition 2.9]{baishya}.

(c) $\Longrightarrow$ (d) It follows from \cite[p. 135]{hall}.

(d) $\Longrightarrow$ (a) It follows from Lemma \ref{c1234}.
\end{proof}

\begin{prop}\label{prop9765}
For a group $G$, $\frac{G}{Z(G)} \cong D_8$ if and only if $G$ is isoclinic with  $D_{16}$.  
\end{prop}

\begin{proof}
By Remark \ref{remark}, $G$ is isoclinic with a finite group $M$ such that $Z(M) \subseteq M'$. In the present scenario, using Lemma \ref{c1234}, we have $\mid M \mid=16$ and consequently, $G$ is isoclinic with  $D_{16}$. Converse is trivial.
\end{proof}

\begin{cor}\label{xxx}
Let $H$ be a non-abelian subgroup of  $G$ with $\frac{G}{Z(G)} \cong D_{8}$. Then $H$ is isoclinic with $Q_8$ or $D_{16}$; and 
 $\mid \Cent(H)\mid=4$ or $6$ respectively.
\end{cor}

\begin{proof}
Combining \cite[Propositions 3.1, 3.13]{baishyaS}, Proposition \ref{prop9765} and \cite[Lemma 3.2]{non} we get the result, by noting that $G$ is a CA-group.
\end{proof}

Using arguments similar to Proposition \ref{prop9765} and GAP \cite{gap}, we also have the following result.

\begin{prop}\label{bcc}
For a group $G$,  $\frac{G}{Z(G)} \cong D_{12}$ if and only if $G$ is isoclinic with  $D_{24}$.  
\end{prop}

\begin{cor}\label{nnnn}
Let $H$ be a non-abelian subgroup of  $G$ with $\frac{G}{Z(G)} \cong D_{12}$. Then $H$ is isoclinic with $Q_8, S_3$ or $D_{24}$; and 
 $\mid \Cent(H)\mid=4, 5$ or $8$ respectively.
\end{cor}

\begin{proof}
Combining \cite[Propositions 3.1, 3.13]{baishyaS}, Proposition \ref{prop9765}, Corollary \ref{n3987} and \cite[Lemma 3.2]{non} we get the result, by noting that $G$ is a CA-group.
\end{proof}

\section{Groups with $\frac{G}{Z(G)} \cong A_4, S_4$ or $A_5$}

Given any finite group $G$ with $\frac{G}{Z(G)} \cong A_4$, the author of \cite{ctc09} obtained $\mid \Cent(G)\mid=6$ or $8$. Here we give some additional information concerning such groups.
Note that for $G=SL(2, 3)$, we have $\frac{G}{Z(G)} \cong A_4$ and $G'=Q_8$. 

\begin{prop}\label{co133}
Let $G$ be any  group such that $\frac{G}{Z(G)} \cong A_4$ and  $G'$ be abelian. Then
\begin{enumerate}
	\item $G' \cong C_2 \times C_2$.
	\item $G' \cap Z(G) = \lbrace 1 \rbrace$.
	\item $G$ is isoclinic with $A_4$.
	\item $G$ is a primitive $\C_{6}$-group.
	\item $\omega(G)=\omega(\frac{G}{Z(G)})=5$.
	\item $C(x)$ for any $x \in G' \setminus Z(G)$ is the proper normal subgroup of $G$.
	\item $\mid G' \mid = \mid \frac{C(x)}{Z(G)}\mid$ for any $x \in G' \setminus Z(G)$.
\end{enumerate} 
\end{prop}

\begin{proof}
By Remark \ref{remark}, $G$ is isoclinic with a finite group  $M$ such that $Z(M) \subseteq M'$ and $\mid G' \cap Z(G) \mid=\mid M' \cap Z(M) \mid$. 

a) We have 
$(\frac{M}{Z(M)})' =\frac{M'Z(M)}{Z(M)} \cong \frac{M'}{M' \cap Z(M)}  \cong C_2 \times C_2$. Now, suppose $a \in M' \setminus Z(M)$. Then $ \frac{C(a)}{Z(M)} \cong C_2 \times C_2$ and $C(a)$ is an abelian normal subgroup of $M$ of index $3$. Therefore by \cite[Lemma 4, p. 303]{zumud}, we have $ G' \cong M' \cong C_2 \times C_2$.

b) It follows from the proof of part (a) that $G' \cap Z(G) = \lbrace 1 \rbrace$.

c) It follows from  \cite[pp. 135]{hall}.

d) It follows from \cite[Lemma 3.2]{non}.

e) It follows from \cite[Lemma 2.1]{non} and \cite[Lemma 2.6]{ed09}.

f) It follows from the fact that for any $x \in G' \setminus Z(G)$, $\frac{C(x)}{Z(G)}$ is the proper normal subgroup of $\frac{G}{Z(G)}$.

g) It follows from part (a).
\end{proof}

\begin{lem}\label{npcor127}
Let $G$ be any CA-group. If $G'$ is non-abelian, then $(\frac{G}{Z(G)})' \cong \frac{G'}{Z(G')}$.
\end{lem}

\begin{proof}
It follows from \cite[Proposition 3.1]{baishyaS}.
\end{proof}

\begin{prop}\label{co188}
Let $G$ be any  group such that $\frac{G}{Z(G)} \cong A_4$ and  $G'$ be non-abelian. Then
\begin{enumerate}
	\item $G' \cong Q_8$.
    \item $G' \cap Z(G)=Z(G') \cong C_2$.
    \item $G $ is isoclinic with $SL(2, 3)$.
	\item $G$ is a $\C_{8}$-group.
	\item $\omega(G)=7$.
	\item $C(x) \ntrianglelefteq G$ for any $x \in G \setminus Z(G)$.
	
\end{enumerate} 
\end{prop}

\begin{proof}
By Remark \ref{remark}, $G$ is isoclinic with a finite group  $M$ such that $Z(M) \subseteq M'$ and $\mid G' \cap Z(G) \mid=\mid M' \cap Z(M) \mid$.  

a) 
In view of  \cite[p. 135]{hall}, we have $G' \cap Z(G)$ is non-trivial. Again, using \cite[Lemma 2.1]{baishya2}, we see that $M$ is a CA-group. In the present scenario, applying \cite[Proposition 3.1]{baishyaS} and  \cite[p. 278 and Proposition 2.17]{schur}, we have $M' \cap Z(M)=Z(M') \cong C_2$. Now, the result follows using Lemma \ref{npcor127} and \cite[Lemma 7]{machale}. 

b) Using \cite[Lemma 2.1]{baishya2} we see that $G$ is a CA-group and therefore by \cite[Proposition 3.1]{baishyaS} and part (a), we have $G' \cap Z(G)=Z(G') \cong C_2$.

c) Using GAP \cite{gap}, we get the result by noting that $\mid Z(M) \mid =2$.

d) In the present scenario, $M$ cannot have a centralizer of index $3$ and consequently, applying \cite[Proposition 2.12]{baishya2} and \cite[Lemma 3.2]{non},  $G$ is a $\C_{8}$-group.

e)  It follows from \cite[Lemma 2.1]{non} and \cite[Lemma 2.6]{ed09}.

f) It follows from the fact that $G$ cannot have a normal centralizer of index $3$.

\end{proof}

\begin{cor}\label{ccc}
Let $H$ be a non-abelian subgroup of $G$ with $\frac{G}{Z(G)} \cong A_4$. Then $H$ is isoclinic with $Q_8, A_4$ or $SL(2, 3)$; and 
 $\mid \Cent(H)\mid=4, 6$ or $8$ respectively.
\end{cor}

\begin{proof}
It follows from \cite[Proposition 3.1]{baishyaS} and Propositions \ref{co133}, \ref{co188} by noting that  $G$ is a CA-group \cite[Lemma 2.1]{baishya2}. 
\end{proof}

We also have the following result which improves  \cite[Theorem 3.5]{en09}.
\begin{cor}\label{ttt}
 $G$ is a primitive $\C_{6}$-group if and only if $G$ is isoclinic with $A_4$.
\end{cor}

\begin{proof}
Suppose $G$ is a primitive $\C_{6}$-group. By \cite[Theorem 3.5]{en09} and Remark \ref{remark}, we have  $\frac{G}{Z(G)} \cong A_4$. Now, the result follows from Propositions \ref{co133} and \ref{co188}. Conversely, 
if $G$ is  isoclinic with $A_4$, then by Proposition \ref{co133}, $G$ is a primitive $\C_{6}$-group. 
\end{proof}

The next two propositions improves \cite[Theorem 1.3]{rostami}.

\begin{prop}\label{c999}
Let $G$ be any  group such that $\frac{G}{Z(G)} \cong S_4$. Then the following are equivalent.

\begin{enumerate}
    \item $G$ is a CA-group.
    \item $\mid G' \cap Z(G) \mid=2$.
    \item $G$ is isoclinic with $GL(2, 3)$.
    \item $G' \cong SL(2, 3)$.
    \item $\omega(G)=13$. 
\end{enumerate} 
\end{prop}

\begin{proof}

By Remark \ref{remark}, $G$ is isoclinic with a finite group  $M$ such that $Z(M) \subseteq M'$ and $\mid G' \cap Z(G) \mid=\mid M' \cap Z(M) \mid$. Now, 

(a) $\Longrightarrow$ (b) If $G$ is a CA-group, then in view of \cite[Proposition 2.1]{baishyaS} and \cite[p. 135]{hall}, we have $G' \cap Z(G)$ is non-trivial. In the present scenario, using \cite[p. 279 and Proposition 2.1.7]{machale}, we have $\mid G' \cap Z(G) \mid=2$.

(b) $\Longrightarrow$ (c) 
If $\mid G' \cap Z(G) \mid=2$, then $\mid M \mid=48$.
Now, using GAP \cite{gap}, we have $G$ is isoclinic with $GL(2, 3)$. 

(c) $\Longrightarrow$ (d)
It follows from the definition of isoclinism.

(d) $\Longrightarrow$ (e) 
If $G' \cong M' \cong  SL(2, 3)$, then $\mid M \mid=48$ and  using GAP \cite{gap}, we have $M$ is isoclinic with $GL(2, 3)$. Now, the result follows 
using \cite[Proposition 2.1]{baishyaS}.

(e) $\Longrightarrow$ (a) Suppose $\omega(G)=13$. In view of \cite[Proposition 2.1]{baishyaS} and \cite[p. 135]{hall}, we have $G' \cap Z(G)$ is non-trivial, by noting that $\omega(S_4)=10$. In the present scenario, 
by \cite[p. 279 and Proposition 2.1.7]{machale}, we have $\mid G' \cap Z(G) \mid=2$,  Therefore using part (b), we have $G$ is isoclinic with $GL(2, 3)$ which is a CA-group. Now, the result follows from \cite[Proposition 2.1]{baishyaS}.
\end{proof}

\begin{prop}\label{b12}
Let $G$ be any  group such that $\frac{G}{Z(G)} \cong S_4$. Then the following are equivalent.

\begin{enumerate}
    \item $G$ is a not a CA-group.
    \item $\mid G' \cap Z(G) \mid=1$.
    \item $G$ is isoclinic with $S_4$.
    \item $G' \cong A_4$.
    \item $\omega(G)=10$. 
    \item $ \mid nacent(G) \mid=4$.
\end{enumerate} 
\end{prop}

\begin{proof}

By Remark \ref{remark}, $G$ is isoclinic with a finite group  $M$ such that $Z(M) \subseteq M'$ and $\mid G' \cap Z(G) \mid=\mid M' \cap Z(M) \mid$. Now, 

(a) $\Longrightarrow$ (b) Suppose $G$ is  not a CA-group. Then in view of \cite[p. 279 and Proposition 2.1.7]{machale} and Proposition \ref{c999}, we have $\mid G' \cap Z(G) \mid=1$.

(b) $\Longrightarrow$ (c) 
It follows from \cite[p. 135]{hall}.

(c) $\Longrightarrow$ (d)
It follows from the definition of isoclinism.

(d) $\Longrightarrow$ (e) 
Suppose $G' \cong M' \cong A_4$. Then $Z(M)$ is trivial and consequently, $G$ is isoclinic with $S_4$. Now, using \cite[Lemma 2.1]{non}, we have 
$\omega(G)=\omega (S_4))=10$.

(e) $\Longrightarrow$ (f)
Using \cite[p. 279 and Proposition 2.1.7]{machale} and Proposition \ref{c999}, we have $\mid G' \cap Z(G) \mid=1$. Consequently, by \cite[p. 135]{hall}, we have $G$ is isoclinic with $S_4$. Now, the result follows from \cite[Proposition 2.1]{baishyaS}.

(f) $\Longrightarrow$ (a) Clearly.

\end{proof}

As an immediate corollary, we have the following result. See \cite[Theorem 1.3]{rostami}.

\begin{cor}\label{co33}
Let $G$ be any  group such that $\frac{G}{Z(G)} \cong S_4$. Then $G$ is a primitive $\C_{14}$-group. Moreover, $\omega(G)=10$ or $13$.
\end{cor}

As another application, we also have the following result.

\begin{cor}\label{co1}
Let $G$ be any  finite solvable CA-group. Then its derived length $dl(G) \leq 4$; with equality if and only if $G$ is isoclinic with $\GL(2, 3)$.
\end{cor}

\begin{proof}
For the first part see \cite[Proposition 3.2(c)]{ctc095}. For the second part, in view of \cite[Proposition 3.2(c)]{ctc095}, the only possibility is  $\frac{G}{Z(G)} \cong S_4$. Now, the result follows from Proposition \ref{c999}.
\end{proof}

The following result  improves \cite[Theorem 1]{ashrafi}.

\begin{prop}\label{ncFf1}
Let $G$ be any group such that  $\frac{G}{Z(G)} \cong A_5$. Then 

\begin{enumerate}
	\item $G$ is a CA-group.
	\item $G' \cong A_5$ or $SL(2, 5)$.
	\item $H \cap Z(G)=Z(H)$ for any non abelian subgroup $H$ of $G$.
	\item $G' \cap Z(G)=Z(G')$  is of order $1$ or $2$.
	\item $G$ is isoclinic with $A_5$ or $SL(2, 5)$.
	\item $\mid \Cent(G)\mid=22$ or $32$.
	\item $\omega(G)=21$ or $31$.
	\item $\mid \Cent(G)\mid=22$ iff $ \mid G' \cap Z(G) \mid=  1 $.
	\item $\mid \Cent(G)\mid=22$ iff $G=Z(G) \times A_5$.
	\item $\mid \Cent(G)\mid=32$ iff $ \mid G' \cap Z(G) \mid = 2$.
	
\end{enumerate} 
\end{prop}

\begin{proof}

By Remark \ref{remark}, $G$ is isoclinic with a finite group  $M$ such that $Z(M) \subseteq M'$.
 
a) It follows from the fact that $ \mid C(xZ(G)) \mid \in \lbrace 3, 4, 5 \rbrace$ for any $xZ(G) \neq Z(G)$.

b) Combining part (a) and \cite[Proposition 2.1]{baishyaS}, we have $M$ is a finite non-solvable CA-group. Now, using \cite[Theorem 3.8]{abc}, we get $G' \cong M' \cong A_5$ or $SL(2, 5)$.

c) It follows from \cite[Proposition 3.1]{baishyaS}.

d) Since $G$ is isoclinic with $M$, therefore by part (c), we have $Z(G')=G' \cap Z(G) \cong M' \cap Z(M)$.  Now, the result follows using \cite[p. 284 and Proposition 2.1.7]{schur}.

e) It follows using part (b) with \cite[Proposition 2.25]{baishyarima}.

f) It follows using part (e), \cite[Lemma 3.2]{non} and GAP \cite{gap}.

g) It follows using parts (a), (f), \cite[Lemma 2.6]{ed09} and \cite[Lemma 2.1]{non}.

h) In view of part (e) and \cite[Lemma 3.2]{non}, we have $\mid \Cent(G)\mid=22$ iff $G$ is isoclinic with $\frac{G}{Z(G)}$. Now, the result follows from \cite[p. 135]{hall}.

i) If $\mid \Cent(G)\mid=22$, then in view of part (h), we have $G' \cong A_5$. Moreover, by \cite[Proposition 3.4]{baishyaS}, we have $G=G'Z(G)$. Therefore using part (h) again, we have $G= Z(G) \times A_5$. Converse is trivial.

j) Combining parts (d), (f) and (h), we get the result.
\end{proof}

\begin{cor}\label{aaa}
Let $H$ be a non-abelian subgroup of $G$ with $\frac{G}{Z(G)} \cong A_5$. Then $H$ is isoclinic with $Q_8, S_3, D_{10}, A_4, SL(2, 3), A_5$ or $SL(2, 5)$. In particular, 
 $\mid \Cent(H)\mid \in \lbrace 4, 5, 6, 7, 8, 22, 32 \rbrace$. 
\end{cor}

\begin{proof}
Using Proposition \ref{ncFf1} and \cite[Proposition 3.1]{baishyaS},  we have 
$
\frac{H}{Z(H)}=\frac{H}{H \cap Z(G)}\cong \frac{HZ(G)}{Z(G)} \leq A_5.
$
Therefore $\frac{H}{Z(H)} \cong C_2 \times C_2, S_3, D_{10}, A_4$ or $A_5$. Now, the result follows using \cite[Proposition 3.13] {baishyaS}, \cite[Corollary 2.5]{baishya},  Corollary \ref{n3987}, and  Propositions \ref{co133}, \ref{co188}, \ref{ncFf1}.
\end{proof}

\section{Groups with $\mid \frac{G}{Z(G)} \mid =p^3$, $p$ a prime}

 Let $p$ be a prime. A finite $p$-group $G$ is said to be a special $p$-group of rank $k$ if $G'=Z(G)$ is elementary abelian of order $p^k$ and $\frac{G}{G'}$ is elementary abelian. In particular, if $G$ is special and $\mid G' \mid=\mid Z(G) \mid=p$, then  $G$ is extraspecial. Recall that a $p$-group of order $p^n> p^3$ is of maximal class, if its nilpotency class is $n-1$.

 For any group $G$, it is known that $\frac{G}{Z(G)}\cong C_p \times C_p$ if and only if $G$ is isoclinic with an extraspecial group of order $p^3$ (see, for example \cite[Proposition 3.13]{baishyaS}). Here we  determine  all groups $G$ with $\mid \frac{G}{Z(G)} \mid = p^3$ upto isoclinism. For such groups, we have  $\mid \Cent(G)\mid=p^2+2$ or $p^2+p+2$ (\cite[Proposition 2.14]{baishya2}).

Note that given a capable group $K$ of order $p^3$, $K \cong C_p \times C_p\times C_p$ if $K$ is abelian. Otherwise, $K \cong D_8$ or a group of exponent $p$ (\cite[Remark 2.13]{baishya2}).  In the following result, which generalizes \cite[Theorem 1 (iv)]{machale}, $G_1=\langle a_1, a_2, a_3, b_{12}, b_{13}, b_{23}; [a_i, a_j]=b_{ij}, a_i^p=a_3^p=b_{ij}^p=1 (1 \leq i<j \leq 3) \rangle$.

\begin{prop}\label{prop34}
Let $G$ be any  group such that $\frac{G}{Z(G)} \cong C_p \times C_p \times C_p$, $p$ a prime. Then
\begin{enumerate}
	\item $\mid \Cent(G)\mid=p^2+2$ if and only if $G$ is isoclinic with a rank $2$ special group of order $p^5$.
	\item $\mid \Cent(G)\mid=p^2+p+2$ if and only if  $G$ is isoclinic with $G_1$.
\end{enumerate} 
\end{prop}

\begin{proof}
 
By  Remark \ref{remark}, $G$ is isoclinic with a finite group $M$ such that $Z(M)= M'$. Now,

a) Suppose $\mid \Cent(G)\mid=p^2+2$. Then using \cite{non}, we have $\mid \Cent(M)\mid=p^2+2$. By \cite[Proposition 2.14]{baishya2} $M$ has an abelian normal centralizer of index $p$. Consequently, using \cite[Theorem 2.3]{baishya} and \cite[Lemma 4.6]{isaacs1}, we have $M'\cong C_p \times C_p$ and  $M$ is a rank $2$ special group of order $p^5$. Conversely, if  $G$ is isoclinic with a rank $2$ special group of order $p^5$, then using  \cite[Theorem 4.1]{ish1}, \cite[Proposition 2.14]{baishya2} and \cite[Lemma 3.2]{non}, we have $\mid \Cent(G)\mid=p^2+2$.\\

b) Suppose $\mid \Cent(G)\mid=p^2+p+2$. Then using \cite{non}, we have $\mid \Cent(M)\mid=p^2+p+2$. In view of \cite[Proposition 2.14]{baishya2}, $M$ is of conjugate type $(1, p^2)$. Therefore using \cite[Theorem 4.1]{ish1}, we have $M'\cong C_p \times C_p \times C_p$ and $M$ is isoclinic with $G_1$, by noting that in the present scenario, $M$ cannot be a Camina group. Conversely, suppose $G$ is isoclinic with $G_1$, which is of conjugate type $(1, p^2)$ (\cite[Theorem 4.1]{ish1}). Then the result follows using \cite[Proposition 2.14]{baishya2} and \cite[Lemma 3.2]{non}.
\end{proof}

As a corollary we have the following generalization of \cite[Theorem 1 (iv)]{machale}, by noting that in the present scenario, we have $G' \cong M'$.

\begin{cor}\label{coro222}
Let $G$ be any  group such that $\frac{G}{Z(G)} \cong C_p \times C_p \times C_p$, $p$ a prime. Then  $G'\cong C_p \times C_p$ or $C_p \times C_p \times C_p$.
\end{cor}

We now consider the groups $G$ for which $\frac{G}{Z(G)}$ is non-abelian of order $p^3$. The case for $p=2$ have been settled in Proposition \ref{prop9765}. In the following result, $G_2=\langle a_1, a_2, b, c_1, c_2; [a_1, a_2]=b, [a_i, b]=c_i, a_i^p=b^p=c_i^p=1  (i=1, 2) \rangle$, $p > 2$.

\begin{prop}\label{prop987}
Let $G$ be any  group and $p$ an odd prime.  Suppose that $\frac{G}{Z(G)}$ is non-abelian of order $p^3$. Then
\begin{enumerate}
	\item $\mid \Cent(G)\mid=p^2+2$ if and only if $G$ is isoclinic with a group of order $p^4$ and maximal class.
	\item $\mid \Cent(G)\mid=p^2+p+2$ if and only if  $G$ is isoclinic with $G_2$.
\end{enumerate} 
\end{prop}

\begin{proof}

By  Remark \ref{remark}, $G$ is isoclinic with a finite group $M$ such that $Z(M) \subseteq M'$.  Now,

a) Suppose $\mid \Cent(G)\mid=p^2+2$. Then using \cite{non}, we have $\mid \Cent(M)\mid=p^2+2$. By \cite[Proposition 2.14]{baishya2} $M$ has an abelian normal centralizer, say, $C(x)$ of index $p$ for some $ x \in M$. Consequently, using \cite[Theorem 2.3]{baishya} we have $\mid M'\mid =p^2$. Hence the result follows. Conversely, suppose $G$ is isoclinic with a group $H$ of order $p^4$ and maximal class. By \cite[Theorem 4.2]{ish1}, $H$ must have a centralizer of index $p$. Now, using \cite[Proposition 2.14]{baishya2} and \cite[Lemma 3.2]{non}, we have $\mid \Cent(G)\mid=p^2+2$.\\

b) Suppose $\mid \Cent(G)\mid=p^2+p+2$.  Then using \cite{non}, we have $\mid \Cent(M)\mid=p^2+p+2$. In view of \cite[Proposition 2.14]{baishya2}, $M$ is of conjugate type $(1, p^2)$. Therefore using \cite[Theorem 4.2]{ish1}, we have $\mid M' \mid=p^3$ and $M$ is isoclinic with $G_2$. Conversely, suppose $G$ is isoclinic with $G_2$, which is of conjugate type $(1, p^2)$ (\cite[Theorem 4.2]{ish1}). Then the result follows using \cite[Proposition 2.14]{baishya2} and \cite[Lemma 3.2]{non}.
\end{proof}

In view of \cite[Propositions 2.1, 3.1]{baishyaS} and \cite[Proposition 2.14]{baishya2}, we now have:

\begin{cor}\label{kkk}
Let $H$ be a non-abelian subgroup of  $G$ with $\mid \frac{G}{Z(G)} \mid =p^3$, $p$ a prime. Then $H$ is isoclinic with any one group in Propositions \ref{prop34},  \ref {prop987}, \ref{prop9765} or an extraspecial group of order $p^3$. In particular, $\mid \Cent(H)\mid \in \lbrace p+2, p^2+2, p^2+p+2 \rbrace$ and $\omega(H) \in \lbrace p+1, p^2+1, p^2+p+1 \rbrace$.
\end{cor}

\section{$\C_{n}$-groups for $n \leq 9$.}

It can be easily checked that there is no $\C_{n}$-group for $n = 2, 3$.

\begin{thm}\label{th34}
Let $G$ be any  $\C_{n}$-group. Then
\begin{enumerate}
	\item $n=1$ if and only if $G$ is isoclinic with $C_1$.
	\item $n=4$ if and only if  $G$ is isoclinic with $Q_8$.
	\item $n=5$ if and only if $G$ is isoclinic with $S_3$ or a non-abelian group of order $27$.
	\item $n=6$ if and only if $G$ is isoclinic with  $A_4, D_{16}$, an ultraspecial group of order $64$ or a rank $2$ special group of order $32$.
	\item $n=7$ if and only if $G$ is isoclinic with $D_{10}, (C_4, C_5)$ or a non-abelian group of order $125$.
	\item $n=8$ if and only if $G$ is isoclinic with $G_1, SL(2, 3)$ or $D_{16}$.
	\item $n=9$ if and only if $G$ is isoclinic with $C_7 \rtimes C_2, C_7 \rtimes C_3, (C_6, C_7)$ or a non-abelian group of order $343$.
\end{enumerate} 
\end{thm}

\begin{proof}
It follows using \cite[Theorem 3.5]{non}, \cite[Propositions 3.13,  3.15]{baishyaS}, \cite[Theorem 2.6]{baishya1}, Corollaries \ref{n3987}, \ref{ccc}, Propositions \ref{n3}, \ref{prop9765},   \ref{bcc} and \cite[Proposition 3.10, 3.13]{baishyaS} by noting that  $G$ is abelian if and only if it is isoclinic with $C_1$.
\end{proof}

%\section*{Acknowledgment}

%I would like to thank Prof. Mohammad Zarrin for his valuable suggestions and comments on the earlier draft of the paper.


\begin{thebibliography}{33}

\bibitem{ed09}
A. Abdollahi, S. M. J. Amiri and A. M. Hassanabadi,  {\em Groups with specific number of centralizers}, Houst. J. Math., \textbf{33} (1) (2007), 43--57.




\bibitem{abc}
A. Abdollahi, S. Akbari and H. R. Maimani,  {\em Non-commuting graph of a group},  J. Algebra, \textbf{298}  (2006), 468--492.



\bibitem{rostami}
 S. M. J. Amiri, M. Amiri and H. Rostami,  {\em Finite groups determined by the number of element centralizers},  Comm. Alg., \textbf{45} (9)  (2017), 1--7.
 
 
 \bibitem{amiriF}
 S. M. J. Amiri, H. Madadi and H. Rostami,  {\em On F-Groups with the central factor of order $p^4$},  Math. Slovaca, \textbf{67} (5)  (2017), 1147--1154.

 \bibitem{amiriS}
S. M. J. Amiri and H. Rostami,  {\em Centralizers in a group whose central factor is simple},  J. Algebra Appl., \textbf{17} (8) (2018), 01--07.



\bibitem{en09}
 A. R. Ashrafi, {\em On finite groups with a given number of centralizers}, Algebra Colloq., \textbf{7} (2) (2000), 139--146.
 
 
 \bibitem{ctc09}
A. R. Ashrafi, {\em Counting the centralizers of some finite groups}, Korean  J. Comput. Appl. Math., \textbf{7} (1) (2000), 115--124.
 
 \bibitem{ashrafi}
A. R. Ashrafi and Bijan Taeri, {\em On finite groups with a certain number of centralizers}, J.  Appl. Math. Comput., \textbf{17} (1--2) (2005), 217--227.
 
 
 
 
\bibitem{baishya}
S. J. Baishya, {\em On  finite groups with specific number of centralizers}, Int. Elect. J. Algebra, \textbf{13} (2013) 53--62.



\bibitem{baishya2}
S. J. Baishya, {\em On  capable groups of order $p^2q$}, Comm. Alg., \textbf{48} (6)  (2020), 2632--2638.


\bibitem{baishya1}
S. J. Baishya, {\em On  finite groups with nine centralizers}, Boll. Unione Mat. Ital., \textbf{9} (2016) 527--531.


\bibitem{baishyaF}
S. J. Baishya, {\em Counting centralizers and $z$-classes of some F-groups}, Comm. Alg., \textbf{50} (6) (2022) 2476--2487.



\bibitem{baishyarima}
S. J. Baishya, {\em Characterizations of some groups in terms of centralizers}, Results Math, \textbf{77: 168}  (2022) https://doi.org 10.1007/s00025-022-01687-4


\bibitem{baishyaS}
S. J. Baishya, {\em On groups with same number of centralizers}, Comm. Alg., \textbf{51} (10) (2023) 4418--4426.

\bibitem{baer}
R. Baer, {\em Groups with preassigned central and central quotient group}, Trans. Amer. Math. Soc., \textbf{44} (1938) 378--412.


\bibitem{zumud}
Y. G. Berkovich and   E. M. Zhmud$^{\prime}$,   \textit{Characters of Finite Groups}. Part 1, Transl. Math. Monographs {\bf 172}, Amer. Math. Soc., Providence, RI, 1998. 


\bibitem{capable}
F. R. Beyl,U. Felgner and P. Schmid {\em On  groups occuring as centre factor groups}, J. algebra., \textbf{61} (1979) 161--177.



\bibitem{cipolla}
M. Cipolla, {\em Sulla struttura dei gruppi di ordine finito}, Rend. Acc. Sci.Fis. Mat. Napoli \textbf{15} (13) (1909), 44--54.




\bibitem{ctc095}
S. Dolfi, M. Herzog and E. Jabara, {\em Finite groups whose noncentral commuting elements have centralizers of equal size}, Bull. Aust. Math Soc.,  \textbf{82}  (2010), 293--304.






\bibitem{hall}
P. Hall, {\em The classification of prime power groups},  J. reine angew Math.,  \textbf{182} (1940), 130--141.

\bibitem{heffernan}
R. Heffernan, D. MachaleA. N. She, {\em Central factor groups and commutativity}, Proc. R. Ir. Acad.,  \textbf{117 A} (2017), 63--75.




\bibitem{isaacs1}
I. Isaacs, {\em Finite group theory}, American Mathematical Society, Providence, (2008).



\bibitem{ish}
K. Ishikawa, {\em On finite $p$-groups which have only two conjugacy lengths}, Israel J. Math.,  \textbf{129} (2002), 119--123.


\bibitem{ish1}
K. Ishikawa, {\em Finite $p$-groups upto isoclinism, Which have only two conjugacy lengths},  J. Algebra,  \textbf{220} (1999), 333--345.


\bibitem{ito}
N. Ito, {\em On finite groups with given conjugate type, I}, Nagoya J. Math.,  \textbf{6} (1953), 17--28.


\bibitem{schur}
G. Karpilovsky, {\em The Schur Multiplier}, Oxford Science, (1987).









\bibitem{pL95}
P. Lescot, {\em  Isoclinism classes and commutativity degrees of finite groups},  J. of Algebra, {\bf 177}
(1995), 847--869.





\bibitem{machale}
D. MacHale and P. O. Murchu, {\em Commutator subgroups of groups with small central factor groups}, Proc. R. Ir. Acad.,  \textbf{93 A} (1) (1993), 123--129.








\bibitem{non}
M. Zarrin, {\em On non-commuting sets and centralizers in infinite groups}, Bull. Aust. Math. Soc. \textbf{93} (2016), 42--46.


\bibitem{nca}
M. A. Iranmanesh and M. H. Zareian, {\em On $n$-Centralizer CA-groups}, Comm. Alg., \textbf{49} (10)  (2021), 4186--4195.

\bibitem{gap}
The GAP Group, GAP- Groups, Algorithoms and Programming, Version 4.12.2, http://www.gap-system.org, 2022. 
 


\end{thebibliography}
\end{document}